\documentclass[11pt,a4paper,openany]{amsart}
\usepackage{mathrsfs}

\usepackage{amsmath,amsfonts,ams,verbatim}
\usepackage{latexsym}
\usepackage{amssymb,leftidx}
\usepackage{extarrows}
\usepackage{overpic}
\usepackage{color}
\usepackage{epsfig}
\usepackage{subfigure}

\makeatletter
\@addtoreset{equation}{section}\makeatother

\numberwithin{equation}{section}
\newtheorem{theorem}{Theorem}[section]
\newtheorem{proposition}[theorem]{Proposition}
\newtheorem{lemma}[theorem]{Lemma}
\theoremstyle{remark}\newtheorem{remark}[theorem]{Remark}

\begin{document}
\title[Strichartz estimates and nonlinear wave equation]{Strichartz estimates and nonlinear wave equation on nontrapping
asymptotically conic manifolds}

\author{Junyong Zhang}
\address{Department of Mathematics, Beijing Institute of Technology, Beijing 100081 China,
and Department of Mathematics, Australian National University,
Canberra ACT 0200, Australia} \email{zhang\_junyong@bit.edu.cn} \maketitle

\begin{abstract} We prove the global-in-time Strichartz estimates for wave equations on the nontrapping
asymptotically conic manifolds. We obtain estimates for the full set
of wave admissible indices, including the endpoint. The key points
are the properties of the microlocalized spectral measure of
Laplacian on this setting showed in \cite{HZ} and a Littlewood-Paley
squarefunction estimate. As applications, we prove the global
existence and scattering for a family of nonlinear wave equations on
this setting.
\end{abstract}

\begin{center}
 \begin{minipage}{120mm}
 \end{minipage}
 \end{center}
\section{Introduction and Statement of Main Results}
Let $(M^\circ, g)$ be a Riemannian manifold of dimension $n\geq2$,
and let $I\subset\R$ be a time interval. Suppose $u(t,z)$: $I\times
M^\circ\rightarrow\mathbb{C}$ to be the solutions of the wave
equation
\begin{equation*}
\partial_{t}^2u+\mathrm{H} u=0, \quad u(0)=u_0(z), ~\partial_tu(0)=u_1(z)
\end{equation*}
where $\mathrm{H}=-\Delta_g$ denotes the minus Laplace-Beltrami
operator on $(M^\circ, g)$. The general homogeneous Strichartz
estimates read
\begin{equation*}
\|u(t,z)\|_{L^q_tL^r_z(I\times M^\circ)}\leq
C\big(\|u_0\|_{H^s(M^\circ)}+\|u_1\|_{H^{s-1}(M^\circ)}\big),
\end{equation*}
where $H^s$ denotes the $L^2$-Sobolev space over $M^\circ$, and
$2\leq q,r\leq\infty$ satisfy
\begin{equation*}
s=n(\frac12-\frac 1r)-\frac1q,\quad
\frac{2}q+\frac{n-1}r\leq\frac{n-1}2,\quad (q,r,n)\neq(2,\infty,3).
\end{equation*}

In the flat Euclidean space, where $M^\circ=\R^n$ and
$g_{jk}=\delta_{jk}$, one can take $I=\R$; see Strichartz
\cite{Str}, Ginibre and Velo \cite{GV}, Keel and Tao \cite{KT}, and
references therein. In general manifolds, for instance the compact
manifold with or without boundary, most of the Strichartz estimates
are local in time. If $M^\circ$ is a compact manifold without
boundary, due to finite speed of propagation one usually works in
coordinate charts and establishes local Strichartz estimates for
variable coefficient wave operators on $\R^n$. See for examples
\cite{Kap, Hart1, Tataru}. Strichartz estimates also are considered
on compact manifold with boundary, see \cite{BLP}, \cite{BHS} and
references therein. When we consider the noncompact manifold with
nontrapping condition, one can obtain global-in-time Strichartz
estimates. For instance, when $M^\circ$ is a exterior manifold in
$\mathbb{R}^n$ to a convex obstacle, for metrics $g$ which agree
with the Euclidean metric outside a compact set with nontrapping
assumption, the global Strichartz estimates are obtained by
Smith-Sogge \cite{HS} for odd dimension, and Burq \cite{Burq} and
Metcalfe \cite{Met} for even dimension. Blair-Ford-Marzuola
\cite{BFM} established global Strichartz estimates for the wave
equation on flat cones $C(\mathbb{S}_{\rho}^1)$ by using the
explicit representation of the fundamental solution.\vspace{0.2cm}

In this paper, we consider the establishment of global-in-time
Strichartz estimates on asymptotically conic manifolds satisfying
a nontrapping condition. Here,  `asymptotically conic' is meant in the
sense that $M^\circ$ can be compactified to a manifold with boundary
$M$ such that $g$ becomes a scattering metric on $M$. On the nontrapping asymptotically conic manifolds, Hassell,
Tao, and Wunsch first established an $L^4_{t,z}$-Strichartz
estimate for  Schr\"odinger equation in \cite{ HTW1} and then they \cite{HTW} extended the
estimate to full admissible Strichartz exponents except endpoint
$q=2$. More precisely, they obtained the local-in-time Strichartz inequalities for non-endpoint
Schr\"odinger admissible pairs $(q,r)$
\begin{equation*}
\|e^{it\Delta_g}u_0\|_{L^q_tL^r_z([0,1]\times M^\circ)}\leq
C\|u_0\|_{L^2(M^\circ)}.
\end{equation*}
Recently, Hassell and the author \cite{HZ} improved the Strichartz
inequalities by replacing the interval $[0,1]$ by $\R$. The purpose
of this article is to extend the above investigations carried out
for Schr\"odinger to wave equations.\vspace{0.2cm}

Let us recall the asymptotically conic geometric setting
(i.e. scattering manifold), which is the same as in
\cite{GHS1,GHS2,HW1,HTW,HZ}. Let $(M^\circ,g)$ be a complete
noncompact Riemannian manifold of dimension $n\geq2$ with one end,
diffeomorphic to $(0,\infty)\times Y$ where $Y$ is a smooth compact
connected manifold without boundary. Moreover, we assume
$(M^\circ,g)$ is asymptotically conic which
means that $M^\circ$ allows a compactification $M$ with
boundary, with $\partial M=Y$, such that the metric $g$ becomes an
asymptotically conic metric on $M$. In details, the metric $g$ in a
collar neighborhood $[0,\epsilon)_x\times \partial M$ near $Y$ takes
the form of
\begin{equation}\label{metric}
g=\frac{\mathrm{d}x^2}{x^4}+\frac{h(x)}{x^2}=\frac{\mathrm{d}x^2}{x^4}+\frac{\sum
h_{jk}(x,y)dy^jdy^k}{x^2},
\end{equation}
where $x\in C^{\infty}(M)$ is a boundary defining function for
$\partial M$ and $h$ is a smooth family of metrics on $Y$. Here we
use $y=(y_1,\cdots,y_{n-1})$ for local coordinates on $Y=\partial
M$, and the local coordinates $(x,y)$ on $M$ near $\partial M$. Away
from $\partial M$, we use $z=(z_1,\cdots,z_n)$ to denote the local
coordinates. If $h_{jk}(x,y)=h_{jk}(y)$ is independent of $x$, we say
$M$ is perfectly conic near infinity. Moreover if every geodesic
$z(s)$ in $M$ reaches $Y$ as $s\rightarrow\pm\infty$, we say $M$ is
nontrapping. The function $r:=1/x$ near $x=0$ can be thought of as a
``radial" variable near infinity and $y$ can be regarded as the $n-1$
``angular" variables; the metric is asymptotic to the exact
conic metric $((0,\infty)_r\times Y, dr^2+r^2h(0))$ as $r\rightarrow
\infty$. The Euclidean space $M^\circ=\mathbb{R}^n$ is an example of
an asymptotically conic manifold with $Y=\mathbb{S}^{n-1}$ and the
standard metric. However a metric cone itself is not an
asymptotically conic manifold because of its cone point. We remark
that the Euclidean space is a perfectly metric nontrapping cone,
where the cone point is a removable singularity.\vspace{0.2cm}

Let $\dot H^s(M^\circ)={(-\Delta_g)}^{-\frac{s}2}L^2(M^\circ)$ be
the homogeneous Sobolev space over $M^\circ$. Throughout this paper,
pairs of conjugate indices are written as $r, r'$, where
$\frac{1}r+\frac1{r'}=1$ with $1\leq r\leq\infty$. Our main result
concerning Strichartz estimates is the following.

\begin{theorem}[Global-in-time Strichartz estimate]\label{Strichartz} Let $(M^\circ,g)$ be an asymptotically conic non-trapping manifold of dimension
$n\geq3$. Let $\mathrm{H}=-\Delta_g$ and suppose that $u$ is the
solution to the Cauchy problem
\begin{equation}\label{eq}
\begin{cases}
\partial_{t}^2u+\mathrm{H} u=F(t,z), \quad (t,z)\in I\times M^\circ; \\ u(0)=u_0(z),
~\partial_tu(0)=u_1(z),
\end{cases}
\end{equation}
for some initial data $u_0\in \dot H^{s}, u_1\in \dot H^{s-1}$, and
the time interval $I\subseteq\R$, then
\begin{equation}\label{stri}
\begin{split}
&\|u(t,z)\|_{L^q_t(I;L^r_z(M^\circ))}+\|u(t,z)\|_{C(I;\dot H^s(M^\circ))}\\
&\qquad\lesssim \|u_0\|_{\dot H^s(M^\circ)}+\|u_1\|_{\dot
H^{s-1}(M^\circ)}+\|F\|_{L^{\tilde{q}'}_t(I;L^{\tilde{r}'}_z(M^\circ))},
\end{split}
\end{equation}
where the pairs $(q,r), (\tilde{q},\tilde{r})\in [2,\infty]^2$
satisfy the wave-admissible condition
\begin{equation}\label{adm}
\frac{2}q+\frac{n-1}r\leq\frac{n-1}2,\quad (q,r,n)\neq(2,\infty,3).
\end{equation}
and the gap condition
\begin{equation}\label{scaling}
\frac1q+\frac nr=\frac n2-s=\frac1{\tilde{q}'}+\frac
n{\tilde{r}'}-2.
\end{equation}
\end{theorem}

\begin{remark}We remark that the estimates are sharp from the
sharpness in \cite{KT} for the Euclidean space. There is no loss of
derivatives. We can take the interval $I=\R$ which means the
estimates are global in time.
\end{remark}

We sketch the proof as follows. Our strategy is to use the abstract
Strichartz estimate proved in Keel-Tao \cite{KT} and our previous argument \cite{HZ} for Schr\"odinger.  Thus,
with $U(t)$ denoting the (abstract) propagator, we need to show
uniform $L^2\rightarrow L^2$ estimate for $U(t)$, and
$L^1\rightarrow L^\infty$ type dispersive estimate on the $U(t)
U(s)^*$ with a bound of the form $O(|t-s|^{-(n-1)/2})$. In the flat
Euclidean setting, the estimates are considerably simpler because of
the explicit formula of the spectral measure. But in our general
setting, the estimates turn out to be more complicated. It follows from
\cite{HW1} that the Schr\"odinger propagator $e^{it\Delta_g}$ fails
to satisfy such a dispersive estimate at any pair of conjugate
points $(z,z') \in M^\circ \times M^\circ$ (i.e. pairs $(z,z')$
where a geodesic emanating from $z$ has a conjugate point at $z'$),
so we need localize the propagator such that the conjugating points
are separated. One may avoid the conjugated points in a
sufficiently short time by using the finite speed of propagation $U(t)(z,z')$. If we do this, we would only obtain the
local-in-time Strichartz estimates. We instead overcome the difficulties
caused by conjugate points by microlocalizing the spectral
measure \cite{HZ},  which is in the same spirit of the proof in
\cite{GHS2} of a \emph{restriction estimate} for the spectral
measure, that is, an estimate of the form
$$
\big\| dE_{\sqrt{\mathbf{H}}}(\lambda) \big\|_{L^p(M^\circ) \to
L^{p'}(M^\circ)} \leq C \lambda^{n(\frac1{p} - \frac1{p'}) - 1},
\quad 1 \leq p \leq \frac{2(n+1)}{n+3}.
$$
However, the microlocalized spectral measure $Q_i(\lambda)
dE_{\sqrt{\mathbf{H}}}(\lambda) Q_i(\lambda)^*$ only has a size
estimate in \cite{GHS2}, where $Q_i(\lambda)$ is a member of a
partition of the identity operator in $L^2(M^\circ)$.  To obtain the
dispersive estimate, the authors \cite{HZ} refined the
microlocalized spectral measure by capturing its oscillatory
behavior. Thus we efficiently exploit the oscillation of the
`spectral multiplier' $e^{it\lambda^2}$ and microlocalized spectral
measure to prove the dispersive estimate for Schr\"odinger.
However, the multiplier $e^{it\lambda}$ corresponding to the wave
equation has much less oscillation than the Schr\"odinger multiplier
$e^{it\lambda^2}$ at high frequency, so we need to modify the argument. Because of this, we have to resort to a
Littlewood-Paley squarefunction estimate on this setting. We remark
that the authors \cite{HZ} avoid using the Littlewood-Paley squarefunction estimate in
the Schr\"odinger case. We prove the Littlewood-Paley squarefunction
estimate on this setting by using a spectral multiplier estimate
in Alexopoulos \cite{Alex} and Stein's \cite{Stein} classical
argument involving Rademacher functions. The crucial ingredient is
to obtain the Gaussian upper bounds on the heat kernel on this
setting. We show the Gaussian upper bounds on the heat kernel by
using the local-in-time heat kernel bounds in Cheng-Li-Yau
\cite{CLY}, and Guillarmou-Hassell-Sikora's \cite{GHS2} restriction
estimate for low frequency which implies the long-time bounds. Having the squarefunction estimate, we
reduce Theorem \ref{Strichartz}
to prove a frequency-localized estimate. To do this, we define a
microlocalized half-wave propagator and prove that it satisfies $L^2\rightarrow L^2$-bounded
and dispersive estimate. We prove the homogeneous Strichartz
estimates for the microlocalized half-wave propagator by using a
semiclassical version of Keel-Tao's argument. The Strichartz
estimate for $e^{it\sqrt{\mathrm{H}}}$ then follows by summing
each microlocalizing piece. The inhomogeneous Strichartz estimates follow from the
homogeneous estimates and the Christ-Kiselev lemma. Compared with
the establishment of Schr\"odinger inhomogeneous Strichartz estimate
in \cite{HZ}, we do not require additional argument since one must have $q>\tilde{q}'$ if both
$(q,r)$ and $(\tilde{q},\tilde{r})$ satisfy \eqref{adm} and
\eqref{scaling}. \vspace{0.2cm}

As an application of the Strichartz estimates, we note that these
inequalities can be utilized to generalize a theorem of
Lindblad-Sogge \cite{LS} on the asymptotically conic non-trapping
manifolds. More precisely, we prove the well-posedness and
scattering of the following semi-linear wave equation,
\begin{equation}\label{neq}
\begin{cases}
\partial_t^2 u+\mathrm{H} u=\gamma|u|^{p-1}u,\qquad
(t,z)\in\R\times M^\circ, \gamma\in\{1,-1\}, \\
u(t,z)|_{t=0}=u_0(z),\quad
\partial_tu(t,z)|_{t=0}=u_1(z).
\end{cases}
\end{equation}
In the case of flat Euclidean space, there are many results on the
understanding of the global existence and scattering. We refer the
readers to \cite{LS,Sogge} and references therein.
Blair-Ford-Marzuola \cite{BFM} also considered similar results for
the wave equation on flat cones $C(\mathbb{S}_{\rho}^1)$. Due to
better understanding the spectral measure, we can extend the result to high dimension.
We here are mostly interested in the range of exponents
$p\in[p_{\text{conf}},1+\frac4{n-2}]$ and the initial data is in $
\dot H^{s_c}(M^\circ)\times \dot H^{s_c-1}(M^\circ)$, where
$p_{\text{conf}}=1+\frac4{n-1}$ and $s_c=\frac
n2-\frac2{p-1}$.\vspace{0.2cm}

Our main result concerning well-posedness and scattering is the
following.

\begin{theorem}\label{thm1}  Let $(M^\circ,g)$ be a non-trapping asymptotically conic manifold of dimension
$n\geq3$. Suppose $p\in[p_{\mathrm{conf}},1+\frac4{n-2}]$ and
$(u_0,u_1)\in \dot H^{s_c}(M^\circ)\times \dot H^{s_c-1}(M^\circ)$,
then there exist $T>0$ and  a unique solution $u$ to \eqref{neq}
satisfying
\begin{equation}\label{lwp} u\in C_t([0,T]; \dot H^{s_c}(M^\circ))\cap
L^{q_0}([0,T];L^{q_0}(M^\circ)),\end{equation} where
$q_0=(p-1)(n+1)/2$. In addition, if there is a small constant
$\epsilon(p)$ such that
\begin{equation}\label{small}
\|u_0\|_{ \dot H^{s_c}}+\|u_1\|_{\dot H^{s_c-1}}<\epsilon(p),
\end{equation}
then there is a unique global and scattering solution $u$ to
\eqref{neq} satisfying
\begin{equation}\label{gwp} u\in C_t(\R; H^{s_c}(M^\circ))\cap
L^{q_0}(\R;L^{q_0}(M^\circ)).\end{equation}
\end{theorem}
\vspace{0.2cm}

This paper is organized as follows. In Section 2 we review the
results of the microlocalized spectral measure and prove the square
function inequalities on this setting. Section 3 is devoted to the
proofs of the microlocalized dispersive estimates and
$L^2$-estimates. In Section 4, we prove the homogeneous and
inhomogeneous Strichartz estimates.  Finally, we apply the
Strichartz estimates to show Theorem \ref{thm1}.\vspace{0.1cm}

{\bf Acknowledgments:}\quad  The author would like to thank
Jean-Marc Bouclet, Andrew Hassell and Changxing Miao for their
helpful discussions and encouragement. He also would like to thank the anonymous referee for
careful reading the manuscript and for giving useful comments. This research was supported
by PFMEC(20121101120044),  Beijing Natural Science Foundation(1144014),
National Natural Science Foundation of China (11401024) and Discovery Grant
DP120102019 from the Australian Research Council.\vspace{0.2cm}

\section{The microlocalized spectral measure and Littlewood-Paley squarefunction estimate}
In this section, we briefly recall the key elements of the
microlocalized spectral measure, which was constructed by Hassell
and the author \cite{HZ} to capture both its size and the
oscillatory behavior. We also prove the Littlewood-Paley
squarefunction estimates on this setting that we require in
subsequence section. \vspace{0.2cm}

\subsection{The microlocalized spectral measure}In the free Euclidean space, the half wave propagator has
an explicit formula by using the Fourier transform, but in the
asymptotically conical manifold it turns out to be quite
complicated. From the results of \cite{GHS1,HW}, we have known that
the Schwartz kernel of the spectral measure can be described as a
Legendrian distribution on the compactification of the space
$M\times M$ uniformly with respect to the spectral parameter
$\lambda$. As pointed out in introduction, we really need to choose
an operator partition of unity to microlocalize the spectral measure
such that the spectral measure can be expressed in a formula
capturing not only the size also the oscillatory behavior. This was
constructed and proved in \cite{HZ}. For convenience, we recall and
slightly modify the statement to adapt our following application.

\begin{proposition}
\label{prop:localized spectral measure} Let $(M^\circ,g)$ and
$\mathrm{H}$ be in Theorem \ref{Strichartz}. For fixed
$\lambda_0>0$, then there exists an operator partition of unity on
$L^2(M)$
\begin{equation}\label{partition}
\begin{split}
\mathrm{Id}=\sum_{i=0}^{N_l}Q^{\mathrm{low}}_i(\lambda)\quad
\text{for}~0<\lambda\leq 2\lambda_0;\\
\mathrm{Id}=\big(\sum_{i=1}^{N'}+\sum_{i=N'+1}^{N_h}\big)Q^{\mathrm{high}}_i(\lambda)\quad
\text{for}~\lambda\geq\lambda_0/2,
\end{split}
\end{equation}
where the $Q^{\mathrm{low}}_i$ and $Q^{\mathrm{high}}_i$ are
uniformly bounded as operators on $L^2$ and $N_l$ and $N_h$ are
bounded independent of $\lambda$, such that

$\bullet$ when $Q(\lambda)$ is equal to either
$Q^{\mathrm{low}}_0(\lambda)$ or $Q^{\mathrm{low}}_1(\lambda)$; or
$Q(\lambda)$ is equal to $Q^{\mathrm{high}}_1(\lambda)$, we have
\begin{equation}\label{bean0}
\begin{split}
\Big|\big(\frac{d}{d\lambda}\big)^{\alpha}\big(Q(\lambda)dE_{\sqrt{\mathrm{H}}}(\lambda){Q}^*(\lambda)\big)\Big|\leq
C_\alpha\lambda^{n-1-\alpha}\quad \forall \alpha\in\mathbb{N}.
\end{split}
\end{equation}

$\bullet$ when $Q(\lambda)$ is equal to
$Q^{\mathrm{low}}_i(\lambda)$ or $Q^{\mathrm{high}}_i(\lambda)$ for
$i\geq2$, we have
\begin{equation}\label{bean}
(Q(\lambda)dE_{\sqrt{\mathrm{H}}}(\lambda)Q^*(\lambda))(z,z')=\lambda^{n-1}e^{\pm
i\lambda d(z,z')}a(\lambda,z,z').
\end{equation}
Here $d(\cdot, \cdot)$ is the Riemannian distance on $M^\circ$, and
$a$ satisfies
\begin{equation}\label{beans}
|\partial_\lambda^\alpha a(\lambda,z,z')|\leq C_\alpha
\lambda^{-\alpha}(1+\lambda d(z,z'))^{-\frac{n-1}2}.
\end{equation}
\end{proposition}

Having this result, we can exploit the oscillations both in the
multiplier $e^{i(t-s)\lambda}$ and in $e^{\pm i\lambda d(z,z')}$ to
obtain the required dispersive estimate for the $TT^*$ version of
the microlocalized propagator.

\subsection{The Littlewood-Paley squarefunction estimate} In this
subsection, we prove the Littlewood-Paley squarefunction estimate
for the asymptotically conic manifold, which allows us to reduce
Theorem \ref{Strichartz} to a frequency-localized estimate (see
Proposition \ref{lStrichartz}).\vspace{0.1cm}

Let $\varphi\in C_0^\infty(\mathbb{R}\setminus\{0\})$ take values in
$[0,1]$ and be supported in $[1/2,2]$ such that
\begin{equation}\label{dp}
1=\sum_{j\in\Z}\varphi(2^{-j}\lambda),\quad\lambda>0.
\end{equation}
Define $\varphi_0(\lambda)=\sum_{j\leq0}\varphi(2^{-j}\lambda)$.
Then the result about the Littlewood-Paley squarefunction estimate
reads as follows:
\begin{proposition}\label{prop:square} Let $(M^\circ,g)$ be an asymptotically conic
manifold, trapping or not, and $\mathrm{H}=-\Delta_g$ is the
Laplace-Beltrami operator on $(M^\circ,g)$. Then for $1<p<\infty$,
there exist constants $c_p$ and $C_p$ depending on $p$ such that
\begin{equation}\label{square}
c_p\|f\|_{L^p(M^\circ)}\leq
\big\|\big(\sum_{j\in\Z}|\varphi(2^{-j}\sqrt{\mathrm{H}})f|^2\big)^{\frac12}\big\|_{L^p(M^\circ)}\leq
C_p\|f\|_{L^p(M^\circ)}.
\end{equation}
\end{proposition}

\begin{remark} To our knowledge, such squarefunction estimates are new in the
case of asymptotically conic manifolds, though the proof is
considerably simpler due to the heat kernel bounds in Cheng-Li-Yau
\cite{CLY}, Guillarmou-Hassell-Sikora's \cite{GHS2} restriction
estimate for low frequency and the spectral multiplier estimates in
Alexopoulos \cite{Alex}. In the general noncompact manifolds with
ends, Bouclet \cite{Bouclet} proved a weak version square function
inequality which was given by for $1<p<\infty$
\begin{equation}\label{hsquare}
\|f\|_{L^p}\lesssim
\big\|\big(\sum_{j\geq0}|\varphi(2^{-2j}\mathrm{H})f|^2\big)^{\frac12}\big\|_{L^p}+\|f\|_{L^2}.
\end{equation}
Bouclet also pointed out that the usual square function inequalities
may fail on asymptotically hyperbolic manifolds and improved
\eqref{hsquare} for asymptotically conic manifolds by showing
\begin{equation}\label{hsquare1}
\|\varphi_0(\mathrm{H})f\|_{L^p}+
\big\|\big(\sum_{j\geq0}|\varphi(2^{-2j}\mathrm{H})f|^2\big)^{\frac12}\big\|_{L^p}\sim
\|f\|_{L^p}.
\end{equation}
One can see that the squarefunction estimate in \eqref{square}
involves the low frequency in contrast to \eqref{hsquare1}.
\end{remark}

\begin{proof}
This proof follows from the Stein's \cite{Stein} classical argument
(in $\R^n$) involving Rademacher functions and an appropriate
Mikhlin-H\"ormander multiplier theorem. Now we provide details as
follows. We notice that the asymptotically conic manifolds are a
relatively well-behaved class of manifolds. In particular, all
section curvatures of $(M^\circ,g)$ approach zero as $x$ goes to
zero, and thus $(M^\circ,g)$ has bounded sectional curvature
and has low bounds for the injectivity radius. Now we need a theorem
in Cheng-Li-Yau \cite{CLY} and recall it for convenience. For
complete Riemannian manifolds $M^\circ$ of bounded sectional
curvature and injectivity radius bounded below, Cheng-Li-Yau's
theorem gives the following local-in-time Gaussian upper bound for
the heat kernel
\begin{lemma} There exist nonzero constants $c$ and $C$ such that
the heat kernel on $M^\circ$, denoted $H(t,z,z')$, satisfies the
Gaussian upper bound of the form for $t\in[0,T]$
\begin{equation}\label{heat}
H(t,z,z')\leq Ct^{-\frac n2}\exp\Big(-\frac{d(z,z')^2}{ct}\Big),
\end{equation}
where $d(z,z')$ is the distance between $z$ and $z'$ on $M^\circ$.
\end{lemma}
We claim that the global-in-time Gaussian upper bound for the heat
kernel also holds, that is
\begin{equation}\label{gb}
H(t,z,z')\lesssim
\frac{1}{|B(z,\sqrt{t})|}\exp\Big(-\frac{d(z,z')^2}{ct}\Big)
\end{equation} holds for all $t>0$, where $|B(z,\sqrt{t})|$ is the volume of the ball of radius $\sqrt{t}$
at $z$. By \eqref{heat}, we only consider the case $t\geq1$. To
prove this, we write
\begin{equation*}
H(t,z,z')=e^{-t\mathrm{H}}(z,z')=\int_0^\infty e^{-t\lambda^2}
dE_{\sqrt{\mathrm{H}}}(\lambda).
\end{equation*}
Choose $\chi \in C_c^\infty(\R)$, such that $\chi(\lambda) = 1$ for
$\lambda \leq 1$, we decompose
\begin{equation*}
\begin{split}
&H(t,z,z')\\&=\int_0^\infty e^{-t\lambda^2}\chi(\lambda)
dE_{\sqrt{\mathrm{H}}}(\lambda)+\int_0^\infty
e^{-t\lambda^2}(1-\chi)(\lambda)
dE_{\sqrt{\mathrm{H}}}(\lambda)\\&=:I+II.
\end{split}
\end{equation*}
By using \cite[Theorem 1.3]{GHS2}, we see for $\lambda\leq 1$
\begin{equation*}
\begin{split}
|dE_{\sqrt{\mathrm{H}}}(\lambda)(z,z')|\leq C \lambda^{n-1}.
\end{split}
\end{equation*}
Hence $I\leq C t^{-\frac n2}$. To treat $II$, we need the following
lemma
\begin{lemma}\label{hp} If the local-in-time heat kernel bound $\|e^{-t\mathrm{H}}\|_{L^1\rightarrow L^2}\leq C t^{-\frac n4}$ holds for $t\leq1$, then the following
spectral projection estimate holds for $\mu\geq1$, $$\|E_{\sqrt{\mathrm{H}}}([0,\mu])\|_{L^1\rightarrow L^2}\leq C\mu^{n/2}.$$
\end{lemma}
\begin{proof} Let
$t=\mu^{-2}$. Notice $1_{[0,\mu]}(s)\leq
e\exp(-\frac {s^2}{\mu^2})$, then spectral projection estimate is proved by writing $E_{\sqrt{\mathrm{H}}}([0,\mu])=E_{\sqrt{\mathrm{H}}}([0,\mu])
e^{\mathrm{H}/\mu^2}e^{-\mathrm{H}/\mu^2}$. Indeed, we have
\begin{equation*}
\begin{split}
\|E_{\sqrt{\mathrm{H}}}([0,\mu])\|_{L^1\rightarrow
L^2}\leq
\|E_{\sqrt{\mathrm{H}}}([0,\mu])e^{\mathrm{H}/\mu^2}\|_{L^2\rightarrow
L^2}\|e^{-\mathrm{H}/\mu^2}\|_{L^1\rightarrow L^2}\leq C\mu^{n/2}.
\end{split}
\end{equation*}

\end{proof}

Now we turn to estimate $II$. From the local-in-time heat kernel
estimate \eqref{heat}, one has $\|e^{-t\mathrm{H}}\|_{L^1\rightarrow L^\infty}\leq C t^{-\frac n2}$ for
$t\leq 1$. By using a $TT^*$ argument, 
$\|e^{-t\mathrm{H}}\|_{L^1\rightarrow L^2}\leq C t^{-\frac n4}$ for
$t\leq 1$. Hence by Lemma \ref{hp}
$\|E_{\sqrt{\mathrm{H}}}([0,\lambda])\|_{L^1\rightarrow L^2}\leq
C\lambda^{n/2}$ for $\lambda\geq1$, which implies
$\|E_{\sqrt{\mathrm{H}}}([0,\lambda])\|_{L^1\rightarrow
L^\infty}\leq C\lambda^{n}$. Therefore we have for $t\geq1$
\begin{equation*}
\begin{split}
\|II\|_{L^1\rightarrow L^\infty}&\leq \sum_{k\geq0}\int_0^\infty
\frac{d}{d\lambda}\left(e^{-t\lambda^2}\phi_k\left(\lambda\right)(1-\chi)(\lambda)\right)
\left\|E_{\sqrt{\mathrm{H}}}(\lambda)\right\|_{L^1\rightarrow
L^\infty}d\lambda\\&\leq Ce^{-t/2}\leq C t^{-\frac n2}.
\end{split}
\end{equation*}
Hence we have proved for all $t>0$
\begin{equation*}
H(t,z,z')\lesssim t^{-\frac n2}.
\end{equation*}
We use a theorem of Grigor'yan \cite[Theorem 1.1]{Grigor} that
establishes Gaussian upper bounds for arbitrary Riemannian
manifolds. His conclusion implies that if $H(t,z,z')$ satisfies
on-diagonal bounds
\begin{equation*}
H(t,z,z)\lesssim t^{-\frac n2},\quad H(t,z',z')\lesssim t^{-\frac
n2},
\end{equation*} then we have
\begin{equation*}
H(t,z,z')\lesssim t^{-\frac n2}\exp\Big(-\frac{d(z,z')^2}{ct}\Big).
\end{equation*}
Since $|B(z,\sqrt{t})|\sim t^{\frac n2}$, this gives
\begin{equation}\label{gb}
H(t,z,z')\lesssim
\frac{1}{|B(z,\sqrt{t})|}\exp\Big(-\frac{d(z,z')^2}{ct}\Big).
\end{equation}
Now we need a result of Alexopoulos \cite[Theorem 6.1]{Alex}, which
outlines how his results on Markov chains can be extended to treat
differential operators on manifolds where the associated heat kernel
satisfies Gaussian upper bounds. We remark here that the
asymptotically conic manifold satisfies the doubling condition in
contrast to the hyperbolic case. Given \eqref{gb}, Alexopoulos'
theorem implies that any spectral multiplier $m(\sqrt{\mathrm{H}})$
satisfying the usual H\"ormander condition maps $L^p(M)\rightarrow
L^p(M)$ for any $p\in (1,\infty)$. Furthermore, this boundedness
holds true for function $m\in C^N(\R)$ which satisfies the weaker
Mihlin-type condition for $N\geq\frac n2+1$
\begin{equation}\label{Minhlin}
\sup_{0\leq k\leq
N}\sup_{\lambda\in\R}\Big|\big(\lambda\partial_\lambda\big)^k
m(\lambda)\Big|\leq C<\infty.
\end{equation}
We now want to apply this result to a family of multipliers
$m^{\pm}(s,\sqrt{\mathrm{H}}), 0\leq s\leq 1$ defined using the
Rademacher functions. Let us introduce the Rademacher functions
defined as follows:

(i) the function $r_0(s)$ is defined by $r_0(s)=1$ on $[0,1/2]$ and
$r_0(s)=-1$ on $(1/2,1)$, and then extended to the real line by
periodicity, i.e. $r_0(s+1)=r_0(s)$;

(ii) for $k\in \N\setminus\{0\}$, $r_k(s)=r_0(2^ks)$.

Given any square integrable sequence of scalars $\{a_k\}_{k\geq0}$,
consider the function $m(s)=\sum_{k\geq0}a_k r_k(s)$. By a lemma in \cite[Appendix D]{Stein}, for any $p\in(1,\infty)$ there exist
constants $c_p$ and $C_p$ such that
\begin{equation}\label{rad}
c_p\|m(s)\|_{L^p([0,1])}\leq
\|m(s)\|_{L^2([0,1])}=\Big(\sum_{k\geq0}|a_k|^2\Big)^{\frac12}\leq
C_p\|m(s)\|_{L^p([0,1])}.
\end{equation}
Now define
$$m^{\pm}(s,\lambda)=\sum_{j\geq0}r_j(s)\varphi_{\pm j}(\lambda)$$
where $\varphi_{\pm j}(\lambda)=\varphi(2^{\mp j}\lambda)$. Then we
define the operator $m^\pm(s,\sqrt{\mathrm{H}})$ through the
spectral measure $dE_{\sqrt{\mathrm{H}}}(\lambda)$:
\begin{equation}
m^\pm(s,\sqrt{\mathrm{H}})=\int_0^\infty m^\pm(s,\lambda)
dE_{\sqrt{\mathrm{H}}}(\lambda).
\end{equation}
We note that this is well-defined by the spectral theory. It can be
verified that $m^\pm(s,\lambda)$ satisfies the condition
\eqref{Minhlin}, and we can take the constant $C$ independent of
$s$. Therefore we have that for $1<p<\infty$ and $f$ in $L^p$ by
\eqref{rad}
\begin{equation*}
\begin{split}
&\Big\|\Big(\sum_{j\geq0}\big|\varphi_{\pm
j}(\sqrt{\mathrm{H}})f\big|^2\Big)^{\frac12}\Big\|^p_{L^p}\lesssim
\Big\|\sum_{j\geq0}\varphi_{\pm
j}(\sqrt{\mathrm{H}})f(z)r_k(s)\Big\|^p_{L^p(M;L^p([0,1]))}\\&\lesssim
\int_{M^\circ}\int_0^1\Big|m^{\pm}(s,\sqrt{\mathrm{H}})f(z)\Big|^pdsdg(z)\lesssim\|f\|^p_{L^p}.
\end{split}
\end{equation*}
Therefore we prove
\begin{equation}\label{lsquare}
\begin{split}
&\Big\|\Big(\sum_{j\in\Z}\big|\varphi_{
j}(\sqrt{\mathrm{H}})f\big|^2\Big)^{\frac12}\Big\|_{L^p}\lesssim\|f\|_{L^p}.
\end{split}
\end{equation}
To see the other inequality, we first define
$\widetilde{\varphi}_j(\lambda)=\sum_{i=j-1}^{j+1}\varphi_i(\lambda)$,
then the above also is true when $\varphi_j(\lambda)$ is replaced by
$\widetilde{\varphi}_j(\lambda)$. Let $f_1\in L^p$ and $f_2\in
L^{p'}$, we see by H\"older's inequality and \eqref{lsquare}
\begin{equation*}
\begin{split}
\Big|\int_{M^\circ}f_1(z)\overline{f_2(z)}d
g(z)\Big|&=\Big|\int_{M^\circ}\sum_{j\in\Z}\big(\widetilde{\varphi}_j(\sqrt{\mathrm{H}})f_1\big)(z)
\overline{\big(\varphi_j(\sqrt{\mathrm{H}})f_2\big)(z)}d g(z)\Big|
\\&\lesssim
\Big\|\big(\sum_{j\in\Z}\big|\widetilde{\varphi}_j(\sqrt{\mathrm{H}})f_1\big|^2\big)^{\frac12}\Big\|_{L^p}
\Big\|\big(\sum_{j\in\Z}\big|\varphi_j(\sqrt{\mathrm{H}})f_2\big|^2\big)^{\frac12}\Big\|_{L^{p'}}
\\&\lesssim
\|f_1\|_{L^p}
\Big\|\big(\sum_{j\in\Z}\big|\varphi_j(\sqrt{\mathrm{H}})f_2\big|^2\big)^{\frac12}\Big\|_{L^{p'}}.
\end{split}
\end{equation*}
By duality, we hence prove \eqref{square}.

\end{proof}

\section{$L^2$-estimates and dispersive estimates}
In this section, we prove the $L^2$-estimates and dispersive
estimates needed for the abstract Keel-Tao argument. We begin by
defining microlocalized propagators and then show the definition makes sense. We do this by showing that each microlocalized
propagator is a bounded operator on $L^2$. This serves both to make
the  definition of each microlocalized propagator allowable, and to
establish the $L^2 \to L^2$ estimate needed for the abstract
Keel-Tao argument. We point out here that the microlocalized
propagators are different from the ones defined in \cite{HZ}, which
allow us to easily show the $L^2 \to L^2$ estimate by spectral
theory on Hilbert space but we need a square function inequality in
the establishment of the Strichartz estimate. Since the
microlocalized propagators avoid the conjugate points, we can prove
the $TT^*$ version dispersive estimates. \vspace{0.2cm}

\subsection{Microlocalized propagator and $L^2$-estimates}
We start by dividing the half wave propagator into a low-energy
piece and a high-energy piece. Choose $\chi \in C_c^\infty(\R)$,
such that $\chi(t) = 1$ for $t \leq 1$. We define
\begin{equation}
U^{\mathrm{low}}(t) = \int_0^\infty e^{it\lambda} \chi(\lambda)
dE_{\sqrt{\mathrm{H}}}(\lambda) , \quad U^{\mathrm{high}}(t) =
\int_0^\infty e^{it\lambda} (1 - \chi)(\lambda)
dE_{\sqrt{\mathrm{H}}}(\lambda).
\end{equation}
Using the partition of unity $1=\sum_{j\in\Z}\varphi(2^{-j}\lambda)$
we define
\begin{equation}
\begin{split}
U^{\mathrm{low}}_j(t) &= \int_0^\infty e^{it\lambda}
\varphi(2^{-j}\lambda)\chi(\lambda) dE_{\sqrt{\mathrm{H}}}(\lambda),\\
U^{\mathrm{high}}_j(t) &= \int_0^\infty e^{it\lambda}
\varphi(2^{-j}\lambda)(1 - \chi)(\lambda)
dE_{\sqrt{\mathrm{H}}}(\lambda).
\end{split}
\end{equation}
Further using the low-energy and high-energy operator partition of
identity operator in Proposition \ref{partition}, we define
\begin{equation}\label{Uij}
\begin{gathered}
U_{i,j}(t) = \int_0^\infty e^{it\lambda}
\varphi(2^{-j}\lambda)\chi(\lambda) Q_i^{\mathrm{low}}(\lambda)dE_{\sqrt{\mathrm{H}}}(\lambda),\quad 0\leq i\leq N_l;\\
U_{i,j}(t) = \int_0^\infty e^{it\lambda} \varphi(2^{-j}\lambda)(1 -
\chi)(\lambda)Q_{i-N_l}^{\mathrm{high}}(\lambda)
dE_{\sqrt{\mathrm{H}}}(\lambda),~ N_l+1\leq i\leq N:= N_l+N_h.
\end{gathered}
\end{equation}
Now we show this definition is unambiguous. To do so, it suffices
to show the above integrals are well defined over any compact
interval in $(0, \infty)$. Suppose that $A(\lambda)$ is a family of
bounded operators on $L^2(M^\circ)$, compactly supported in $[a,b]$
and $C^1$ in $\lambda\in (0,\infty)$. Integrating by parts, the
integral of
$$
\int_a^b A(\lambda) dE_{\sqrt{\mathrm{H}}}(\lambda)
$$
is given by
\begin{equation}\label{mean}
E_{\mathrm{\sqrt{H}}}(b) A(b) - E_{\mathrm{\sqrt{\mathrm{H}}}}(b)
A(a) - \int_a^b \frac{d}{d\lambda} A(\lambda)
E_{\sqrt{\mathrm{H}}}(\lambda) \, d\lambda.
\end{equation}
Now we need the following lemma which is the consequence of
\cite[Lemma 2.3, Lemma 3.1]{HZ}.

\begin{lemma}\label{QQ'}
Each $Q^{\mathrm{low}}_i(\lambda)$ and each operator $\lambda
\partial_\lambda Q^{\mathrm{low}}_i(\lambda)$ is bounded on $L^2(M^\circ)$
uniformly in $\lambda$. The same statements are true for the high
energy operators $Q^{\mathrm{high}}_i(\lambda)$.\end{lemma}

\begin{proof} We use the notation in \cite{GHS1,HZ,HW}.
The uniform boundedness of the scattering pseudodifferential
operator $Q^{\mathrm{low}}_i(\lambda) \in \Psi^{-\infty}_k(M,
M^2_{k,b})$ is straightforward to prove using the fact that the
order is $-\infty$. This implies that the kernel is smooth and
uniformly bounded on iterated blowup space $M^2_{k,\mathrm{sc}}$, as
a multiple of the half density bundle $\Omega_{k,b}^{\frac12}$. This
bundle has a nonzero section given, in the region where $x \leq C
\lambda$, by $\lambda^n |dg dg'|^{1/2} |d\lambda/\lambda|^{1/2}$,
where the $|d\lambda/\lambda|^{1/2}$ is a purely formal factor,
included to make a half-density on the whole space $M^2_{k,b}$,
including in the $\lambda$-direction. On the other hand, the kernels
are chosen to have support in a neighborhood of the diagonal, which
is equivalent to the region where $d(z, z') \leq C \lambda^{-1}$. It
follows that the kernel is bounded by a multiple of the
characteristic function of the set $\{ (z, z') \mid d(z, z') \leq C
\lambda^{-1} \}$ times the Riemannian half-density. Moreover, the
same is true for $\lambda d_\lambda Q_i^{\mathrm{low}}(\lambda)$,
due to the smoothness of the kernel on $M^2_{k,\mathrm{sc}}$. Since
the volume of each ball of radius $r$ on $M^\circ$ is between $c
r^n$ and $C r^n$, Schur's test shows that such kernels are bounded
on $L^2(M^\circ)$ uniformly in $\lambda$.

The high energy operators $Q^{\mathrm{high}}(\lambda)$ are
semiclassical pseudodifferential operators of semiclassical order 0
and differential order $-\infty$. Therefore, they take the form
$$
\lambda^n \int e^{i\lambda(z-z') \cdot \zeta} a(z, \zeta,
\lambda^{-1}) \, d\zeta
$$
in the interior, or
$$
\lambda^n \int e^{i\lambda ((y-y') \cdot \eta + (\sigma - 1)\nu/x}
a(x, y, \eta, \nu, \lambda^{-1}) \, d\eta \, d\nu
$$
near the boundary. Here $a$ is smooth and compactly supported in its
arguments. Integration by parts in $\zeta$, or in $\eta, \nu$, shows
that the kernel is rapidly decreasing in $\lambda |z-z'|$,
respectively $\lambda \sqrt{|y-y'|^2/x^2 + (\sigma - 1)^2/x^2}$.
Equivalently, the kernel is rapidly decreasing in $\lambda d(z,
z')$. We see that the kernel is point-wise bounded by $C \lambda^n (1
+ \lambda d(z, z'))^{-N}$ for any $N$. The same is true for $\lambda
d_\lambda Q_i^{\mathrm{high}}(\lambda)$. Again Schur's test shows
that such kernels are bounded on $L^2(M^\circ)$ uniformly in
$\lambda$.
\end{proof}
In view of this lemma, we can take $A(\lambda) = e^{it\lambda}
\chi(\lambda) \varphi(2^{-j})Q^{\mathrm{low}}_i(\lambda)$ (for $0
\leq i \leq N_l$), or $
e^{it\lambda}\varphi(2^{-j})(1-\chi)(\lambda)
Q^{\mathrm{high}}_{i-N_l}(\lambda)$ (for $N_l+1 \leq i \leq N$),
this means that the integrals are well-defined over any compact
interval in $(0, \infty)$, hence the operators $U_{i,j}(t)$ are
well-defined. Now we see these operators are bounded on $L^2$. We
only consider the low frequency part since a similar argument
also gives the boundedness on $L^2$ for high energy part. We have
for $0\leq i\leq N_l$, by  \cite[Lemma 5.3]{HZ},
\begin{equation}\begin{gathered}
U_{i,j}(t) U_{i, j}(t)^* = \int   \chi(\lambda)^2 \varphi\big(
\frac{\lambda}{2^j} \big) \varphi\big( \frac{\lambda}{2^j} \big)
Q^{\mathrm{low}}_i(\lambda) dE_{\sqrt{\mathrm{H}}}(\lambda) Q^{\mathrm{low}}_i(\lambda)^* \\
= -\int \frac{d}{d\lambda} \Big( \chi(\lambda)^2 \varphi\big(
\frac{\lambda}{2^j} \big) \varphi\big( \frac{\lambda}{2^j} \big)
Q^{\mathrm{low}}_i(\lambda) \Big) E_{\sqrt{\mathrm{H}}}(\lambda) Q^{\mathrm{low}}_i(\lambda)^*  \\
- \int  \chi(\lambda)^2 \varphi\big( \frac{\lambda}{2^j} \big)
\varphi\big( \frac{\lambda}{2^j} \big) Q^{\mathrm{low}}_i(\lambda)
E_{\sqrt{\mathrm{H}}}(\lambda) \frac{d}{d\lambda}
Q^{\mathrm{low}}_i(\lambda)^*.
\end{gathered}\label{Uijk}\end{equation}
We observe that this is independent of $t$ and we also note that the
integrand is a bounded operator on $L^2$, with an operator bound of
the form $C/\lambda$ where $C$ is uniform, as we see from
Lemma~\ref{QQ'} and the support property of $\varphi$. The integral
is therefore uniformly bounded, as we are integrating over a dyadic
interval in $\lambda$. Hence we have shown that
\begin{proposition}[$L^2$-estimates]\label{energy} Let $U_{i,j}(t)$ be defined in \eqref{Uij}.
Then there exists a constant $C$ independent of $t, z, z'$ such that
$\|U_{i,j}(t)\|_{L^2\rightarrow L^2}\leq C$ for all $i\geq 0,
j\in\Z$.
\end{proposition}

\subsection{Dispersive estimates} Next we aim to establish the dispersive estimates for the
microlocalized $U_{i,j}(t)U^*_{i,j}(s)$. We need the following
proposition.
\begin{proposition}[Microlocalized dispersive estimates]\label{prop:dispersive}
Let $Q(\lambda)$ be the operator $Q_i^{\mathrm{low}}$ or
$Q_i^{\mathrm{high}}$ constructed as in Proposition \ref{partition}
and suppose $\phi\in C_c^\infty([1/2, 2])$ and takes value in
$[0,1]$. Then the kernel estimate
\begin{equation}\label{dispersive}
\begin{split}
\Big|\int_0^\infty e^{it\lambda}\phi(2^{-j}\lambda) \big(Q(\lambda)
&dE_{\sqrt{\mathrm{H}}}(\lambda)Q^*(\lambda)\big)(z,z')
d\lambda\Big|\\&\leq C 2^{j(n+1)/2}(2^{-j}+|t|)^{-(n-1)/2}
\end{split}
\end{equation}
holds for a constant $C$ independent of $j\in \Z$ and points
$z,z'\in M^\circ$.
\end{proposition}

\begin{proof} The key to the proof is to apply Proposition \ref{partition}. For $Q=Q_i^{\mathrm{low}}$ for $i=0,1$, or
$Q=Q_1^{\mathrm{high}}$, we have by Proposition \ref{partition}
\begin{equation*}
\Big|\int_0^\infty e^{it\lambda}\phi(2^{-j}\lambda) \big(Q(\lambda)
dE_{\sqrt{\mathrm{H}}}(\lambda)Q^*(\lambda)\big)(z,z')
d\lambda\Big|\leq C 2^{jn}.
\end{equation*}
 We use the $N$-times integration by parts to obtain by
\eqref{bean0}
\begin{equation*}
\begin{split}
&\Big|\int_0^\infty e^{it\lambda}\phi(2^{-j}\lambda) \big(Q(\lambda)
dE_{\sqrt{\mathrm{H}}}(\lambda)Q^*(\lambda)\big)(z,z')
d\lambda\Big|\\&\leq \Big|\int_0^\infty \big(\frac1{
it}\frac\partial{\partial\lambda}\big)^{N}\big(e^{it\lambda}\big)
\phi(2^{-j}\lambda)\big(Q(\lambda)
dE_{\sqrt{\mathrm{H}}}(\lambda)Q^*(\lambda)\big)(z,z')
d\lambda\Big|\\& \leq
C_N|t|^{-N}\int_{2^{j-1}}^{2^{j+1}}\lambda^{n-1-N}d\lambda\leq
C_N|t|^{-N}2^{j(n-N)}.
\end{split}
\end{equation*}
Therefore we obtain
\begin{equation}\label{dispersive1}
\begin{split}
&\Big|\int_0^\infty e^{it\lambda} \phi(2^{-j}\lambda)\big(Q(\lambda)
dE_{\sqrt{\mathrm{H}}}(\lambda)Q^*(\lambda)\big)(z,z')
d\lambda\Big|\leq C_N2^{jn}(1+2^j|t|)^{-N}.
\end{split}
\end{equation}
By choosing $N=(n-1)/2$, we prove \eqref{dispersive}. When $Q$ is
equal to $Q_i^{\mathrm{low}}$ or $Q_i^{\mathrm{high}}$ for $i\geq2$,
we see by Proposition \ref{partition}
\begin{equation*}
\begin{split}
&\Big|\int_0^\infty e^{it\lambda}\phi(2^{-j}\lambda) \big(Q(\lambda)
dE_{\sqrt{\mathrm{H}}}(\lambda)Q^*(\lambda)\big)(z,z')
d\lambda\Big|\\&= \Big|\int_0^\infty \left(\frac1{
i(t-d(z,z'))}\frac\partial{\partial\lambda}\right)^{N}\big(e^{i(t-d(z,z'))\lambda}\big)
\phi(2^{-j}\lambda)\lambda^{n-1}a(\lambda,z,z') d\lambda\Big|\\&
\leq
C_N|t-d(z,z')|^{-N}\int_{2^{j-1}}^{2^{j+1}}\lambda^{n-1-N}(1+\lambda
d(z,z'))^{-\frac{n-1}2}d\lambda\\&\leq
C_N2^{j(n-N)}|t-d(z,z')|^{-N}(1+2^jd(z,z'))^{-(n-1)/2}.
\end{split}
\end{equation*}
It follows that
\begin{equation}\label{dispersive2}
\begin{split}
&\Big|\int_0^\infty e^{it\lambda} \phi(2^{-j}\lambda)\big(Q(\lambda)
dE_{\sqrt{\mathrm{H}}}(\lambda)Q^*(\lambda)\big)(z,z')
d\lambda\Big|\\&\leq
C_N2^{jn}\big(1+2^j|t-d(z,z')|\big)^{-N}(1+2^jd(z,z'))^{-(n-1)/2}.
\end{split}
\end{equation}
If $|t|\sim d(z,z')$, it is clear to see \eqref{dispersive}.
Otherwise, we have $|t-d(z,z')|\geq c|t|$ for some small constant
$c$, then choose $N=(n-1)/2$ to prove \eqref{dispersive}.
\end{proof}\vspace{0.1cm}
\begin{remark} If $N=\frac{n-1}2$ is not an integer, one may need geometric mean argument to modify the proof.
\end{remark}\vspace{0.2cm}

As a consequence of Proposition \ref{prop:dispersive}, we
immediately have
\begin{proposition}\label{prop:Dispersive} Let $U_{i,j}(t)$ be defined in \eqref{Uij}.
Then there exists a constant $C$ independent of $t, z, z'$ for all
$i\geq 0, j\in\Z$ such that
\begin{equation}\label{Dispersive}
\|U_{i,j}(t)U^*_{i,j}(s)\|_{L^1\rightarrow L^\infty}\leq C
2^{j(n+1)/2}(2^{-j}+|t-s|)^{-(n-1)/2}.
\end{equation}
\end{proposition}

\section{Strichartz estimates}
In this section, we show the Strichartz estimates in Theorem
\ref{Strichartz}. To obtain the Strichartz estimates, we need a
variant of Keel-Tao's abstract Strichartz estimate for wave
equation.

\subsection{Semiclassical Strichartz estimates}
We need a variety of the abstract Keel-Tao's Strichartz estimates
theorem. This is an analogue of the semiclassical Strichartz
estimates for Schr\"odinger in \cite{KTZ, Zworski}.

\begin{proposition}\label{prop:semi}
Let $(X,\mathcal{M},\mu)$ be a $\sigma$-finite measured space and
$U: \mathbb{R}\rightarrow B(L^2(X,\mathcal{M},\mu))$ be a weakly
measurable map satisfying, for some constants $C$, $\alpha\geq0$,
$\sigma, h>0$,
\begin{equation}\label{md}
\begin{split}
\|U(t)\|_{L^2\rightarrow L^2}&\leq C,\quad t\in \mathbb{R},\\
\|U(t)U(s)^*f\|_{L^\infty}&\leq
Ch^{-\alpha}(h+|t-s|)^{-\sigma}\|f\|_{L^1}.
\end{split}
\end{equation}
Then for every pair $q,r\in[1,\infty]$ such that $(q,r,\sigma)\neq
(2,\infty,1)$ and
\begin{equation*}
\frac{1}{q}+\frac{\sigma}{r}\leq\frac\sigma 2,\quad q\ge2,
\end{equation*}
there exists a constant $\tilde{C}$ only depending on $C$, $\sigma$,
$q$ and $r$ such that
\begin{equation}\label{stri}
\Big(\int_{\R}\|U(t) u_0\|_{L^r}^q dt\Big)^{\frac1q}\leq \tilde{C}
\Lambda(h)\|u_0\|_{L^2}
\end{equation}
where $\Lambda(h)=h^{-(\alpha+\sigma)(\frac12-\frac1r)+\frac1q}$.
\end{proposition}

\begin{proof} If
$(q,r,\sigma)\neq (2,\infty,1)$ is on the line $\frac1q+\frac\sigma
r=\frac\sigma 2$, we replace $(|t-s|+h)^{-\sigma}$ by
$|t-s|^{-\sigma}$ and then we closely follow Keel-Tao's argument
\cite[Sections 3-7]{KT} to show \eqref{stri}. So we only consider
$\frac1q+\frac\sigma r<\frac\sigma 2$. By the $TT^*$ argument, it
suffices to show
\begin{equation*}
\begin{split}
\Big|\iint\langle U(s)^*f(s), U(t)^*g(t) \rangle dsdt\Big|\lesssim
\Lambda(h)^2\|f\|_{L^{q'}_tL^{r'}}\|g\|_{L^{q'}_tL^{r'}}.
\end{split}
\end{equation*}
By the interpolation of the bilinear form of \eqref{md}, we have
\begin{equation*}
\begin{split}
\langle U(s)^*f(s), U(t)^*g(t) \rangle&\leq
Ch^{-\alpha(1-\frac2r)}(h+|t-s|)^{-\sigma(1-\frac2r)}\|f\|_{L^{r'}}\|g\|_{L^{r'}}.
\end{split}
\end{equation*}
Therefore we see by H\"older's and Young's inequalities for
$\frac1q+\frac\sigma r<\frac\sigma 2$
\begin{equation*}
\begin{split}
\Big|\iint\langle U(s)^*f(s),& U(t)^*g(t) \rangle
dsdt\Big|\\&\lesssim
h^{-\alpha(1-\frac2r)}\iint(h+|t-s|)^{-\sigma(1-\frac2r)}\|f(t)\|_{L^{r'}}\|g(s)\|_{L^{r'}}dtds\\&
\lesssim
h^{-\alpha(1-\frac2r)}h^{-\sigma(1-\frac2r)+\frac2q}\|f\|_{L^{q'}_tL^{r'}}\|g\|_{L^{q'}_tL^{r'}}.
\end{split}
\end{equation*}
This proves \eqref{stri}.
\end{proof}

\subsection{Homogeneous Strichartz estimates} To prove the homogeneous Strichartz estimates, we first
reduce the estimates to frequency localized estimates. Using the
Littlewood-Paley frequency cutoff $\varphi_k(\sqrt{\mathrm{H}})$, we
define
\begin{equation}\label{loc}
u_k(t,\cdot)=\varphi_k(\sqrt{\mathrm{H}})u(t,\cdot).
\end{equation}
Notice the frequency cutoffs commute with the operator
$\mathrm{H}=-\Delta_g$, the frequency localized solutions
$\{u_k\}_{k\in\Z}$ satisfy the family of Cauchy problems
\begin{equation}\label{leq}
\partial_{t}^2u_k+\mathrm{H} u_k=0, \quad u_k(0)=f_k(z),
~\partial_tu_k(0)=g_k(z),
\end{equation}
where $f_k=\varphi_k(\sqrt{\mathrm{H}})u_0$ and
$g_k=\varphi_k(\sqrt{\mathrm{H}})u_1$. By the squarefunction
estimates \eqref{square} and Minkowski's inequality, we obtain for
$q,r\geq2$
\begin{equation}\label{LP}
\|u\|_{L^q(\R;L^r(M^\circ))}\lesssim
\Big(\sum_{k\in\Z}\|u_k\|^2_{L^q(\R;L^r(M^\circ))}\Big)^{\frac12}.
\end{equation}
Let $U(t)=e^{it\sqrt{\mathrm{H}}}$ be the half wave operator, then
we write
\begin{equation}\label{sleq}
\begin{split}
u_k(t,z)
=\frac{U(t)+U(-t)}2f_k+\frac{U(t)-U(-t)}{2i\sqrt{\mathrm{H}}}g_k.
\end{split}
\end{equation}
To prove the homogeneous estimates in Theorem \ref{Strichartz}, that
is $F=0$, it suffices to show by \eqref{LP} and \eqref{sleq}
\begin{proposition}\label{lStrichartz} Let
$f=\varphi_k(\sqrt{\mathrm{H}})f$ for $k\in\Z$, we have
\begin{equation}\label{lstri}
\|U(t)f\|_{L^q_tL^r_z(\mathbb{R}\times M^\circ)}\lesssim
2^{ks}\|f\|_{L^2(M^\circ)},
\end{equation}
where the admissible pair $(q,r)\in [2,\infty]^2$ and $s$ satisfy
\eqref{adm} and \eqref{scaling}.
\end{proposition}
Now we prove this proposition. By using Proposition \ref{energy} and
Proposition \ref{prop:Dispersive}, we have the estimates \eqref{md}
for $U_{i,j}(t)$, where $\alpha=(n+1)/2$, $\sigma=(n-1)/2$ and
$h=2^{-j}$. Then it follows from Proposition \ref{prop:semi} that
\begin{equation*}
\|U_{i,j}(t)f\|_{L^q_t(\R:L^r(M^\circ))}\lesssim
2^{j[n(\frac12-\frac1r)-\frac1q]} \|f\|_{L^2(M^\circ)}.
\end{equation*}
Notice that
\begin{equation*}
U(t)=\sum_{i=0}^{N}\sum_{j\in\Z}U_{i,j}(t),
\end{equation*}
we can write
\begin{equation*}
U(t)f=\sum_{i}\sum_{j\in\mathbb{Z}}\int_0^\infty
e^{it\lambda}\varphi(2^{-j}\lambda)Q_i(\lambda)dE_{\sqrt{\mathrm{H}}}(\lambda)
\widetilde{\varphi}(2^{-j}\sqrt{\mathrm{H}})f
\end{equation*}
where $\widetilde{\varphi} \in C_0^\infty(\R\setminus\{0\})$ takes
values in $[0,1]$ such that $\widetilde{\varphi}\varphi=\varphi$. In
view of the condition $f=\varphi(2^{-k}\sqrt{\mathrm{H}})f$, then
$\widetilde{\varphi}(2^{-j}\sqrt{\mathrm{H}})f$ vanishes if
$|j-k|\gg1$. Hence we obtain
\begin{equation*}
\|U(t)f\|_{L^q_t(\R:L^r(M^\circ))}\lesssim
2^{k[n(\frac12-\frac1r)-\frac1q]} \|f\|_{L^2(M^\circ)},
\end{equation*}
which implies \eqref{lstri}.

\subsection{Inhomogeneous Strichartz estimates}
In this subsection, we prove the inhomogeneous Strichartz estimates
including the endpoint $q=2$ for $n \geq 4$. Let
$U(t)=e^{it\sqrt{\mathrm{H}}}: L^2\rightarrow L^2$. We have already
proved that
\begin{equation}
\|U(t)u_0\|_{L^q_tL^r_z}\lesssim\|u_0\|_{\dot{H}^s}
\end{equation} holds for all $(q,r,s)$ satisfying \eqref{adm} and \eqref{scaling}.
For $s\in\R$ and $(q,r)$ satisfying \eqref{adm} and \eqref{scaling},
we define the operator $T_s$ by
\begin{equation}\label{Ts}
\begin{split}
T_s: L^2_z&\rightarrow L^q_tL^r_z,\quad f\mapsto \mathrm{H}^{-\frac
s2}e^{it\sqrt{\mathrm{H}}}f.
\end{split}
\end{equation}
Then we have by duality
\begin{equation}\label{Ts*}
\begin{split}
T^*_{1-s}: L^{\tilde{q}'}_tL^{\tilde{r}'}_z\rightarrow L^2,\quad
F(\tau,z)&\mapsto \int_{\R}\mathrm{H}^{\frac
{s-1}2}e^{-i\tau\sqrt{\mathrm{H}}}F(\tau)d\tau,
\end{split}
\end{equation}
where $1-s=n(\frac12-\frac1{\tilde{r}})-\frac1{\tilde{q}}$.
Therefore we obtain
\begin{equation*}
\Big\|\int_{\R}U(t)U^*(\tau)\mathrm{H}^{-\frac12}F(\tau)d\tau\Big\|_{L^q_tL^r_z}
=\big\|T_sT^*_{1-s}F\big\|_{L^q_tL^r_z}\lesssim\|F\|_{L^{\tilde{q}'}_tL^{\tilde{r}'}_z}.
\end{equation*}
Since $s=n(\frac12-\frac1r)-\frac1q$ and
$1-s=n(\frac12-\frac1{\tilde{r}})-\frac1{\tilde{q}}$, thus $(q,r),
(\tilde{q},\tilde{r})$ satisfy \eqref{scaling}. By the
Christ-Kiselev lemma \cite{CK}, we thus obtain for $q>\tilde{q}'$,
\begin{equation}\label{non-inhomgeneous}
\begin{split}
\Big\|\int_{\tau<t}\frac{\sin{(t-\tau)\sqrt{\mathrm{H}}}}
{\sqrt{\mathrm{H}}}F(\tau)d\tau\Big\|_{L^q_tL^r_z}\lesssim\|F\|_{L^{\tilde{q}'}_t{L}^{\tilde{r}'}_z}.
\end{split}
\end{equation}
Notice that for all $(q,r), (\tilde{q},\tilde{r})$ satisfy
\eqref{adm} and \eqref{scaling}, we must have $q>\tilde{q}'$.
Therefore we have proved all inhomogeneous Strichartz estimates
including the endpoint $q=2$.

\section{Wellposedness and scattering}
In this section, we prove Theorem \ref{thm1}. We prove the result by
a contraction mapping argument. The key point is the application of
Strichartz estimates. Let $q_0=(n+1)(p-1)/2$, $q_1=2(n+1)/(n-1)$ and
$\alpha=s_c-\frac12$. For any small constant $\epsilon>0$ such that
$2\epsilon<\epsilon(p)$ given by \eqref{small}, there exists $T>0$
such that
\begin{equation}
\begin{split}X:=\Big\{u: ~&u\in C_t(\dot H^{s_c})\cap L^{q_0}([0,T];L^{q_0}(M^\circ))\cap L^{q_1}([0,T];\dot
H^{\alpha}_{q_1}(M^\circ)),\\&
\|u\|_{L^{q_0}([0,T];L^{q_0}(M^\circ))}+\|u\|_{L^{q_1}([0,T];\dot
H^{\alpha}_{q_1}(M^\circ))}\leq
C\epsilon\Big\}.\end{split}\end{equation} Consider the solution map
$\Phi$ defined by
\begin{equation*}
\begin{split}
\Phi(u)&=\cos(t\sqrt{\mathrm{H}})u_0(z)+\frac{\sin(t\sqrt{\mathrm{H}})}{\sqrt{\mathrm{H}}}u_1(z)
+\int_0^t\frac{\sin\big((t-s)\sqrt{\mathrm{H}}\big)}{\sqrt{\mathrm{H}}}F(u(s,z))\mathrm{d}s
\\&=:u_{\text{hom}}+u_{\text{inh}},
\end{split}
\end{equation*}
where $F(u)=\gamma |u|^{p-1}u$. We claim the map $\Phi: X\rightarrow
X$ is contracting. Indeed, by Theorem \ref{Strichartz}, we obtain
\begin{equation}\label{5.4}
\begin{split}
\|u_{\text{hom}}\|_{C_t(\dot H^{s_c})\cap
L^{q_0}(\R;L^{q_0}(M^\circ))\cap L^{q_1}(\R;\dot
H^{\alpha}_{q_1}(M^\circ))}\leq C\big(\|u_0\|_{\dot
H^{s_c}}+\|u_1\|_{\dot H^{s_c-1}}\big).
\end{split}
\end{equation}
Hence we must have
\begin{equation}\label{5.4}
\begin{split}
\|u_{\text{hom}}\|_{ L^{q_0}([0,T];L^{q_0}(M^\circ))\cap
L^{q_1}([0,T];\dot H^{\alpha}_{q_1}(M^\circ))}\leq \frac12C\epsilon
\end{split}
\end{equation}
for $T=\infty$ if the initial data has small norm $\epsilon(p)$, or,
if not, this inequality will be satisfied for some $T>0$ by the
dominated convergence theorem. We first note that the Sobolev
embedding $ L^{q_0}_t\dot H^{\alpha}_{r_0}\hookrightarrow
L_{t,z}^{q_0}$ where $r_0=2n(n+1)(p-1)/[(n^2-1)(p-1)-4]$. Under the
condition $p\in [p_{\mathrm{conf}}, 1+\frac{4}{n-2}]$, it is easy to
check that the pairs $(q_0,r_0), (q_1,q_1)$ satisfy \eqref{adm} and
\eqref{scaling} with $s=1/2$. Applying Theorem \ref{Strichartz} with
$\tilde{q}'=\tilde{r}'=\frac{2(n+1)}{n+3}$, one has
\begin{equation}\label{5.5}
\begin{split}
\|u_{\text{inh}}\|_{C_t(\dot H^{s_c})\cap
L^{q_0}([0,T];L^{q_0}(M^\circ))\cap L^{q_1}([0,T];\dot
H^{\alpha}_{q_1}(M^\circ))}\leq C\|F(u)\|_{L^{\tilde{q}'}_t\dot
H^{\alpha}_{\tilde{r}'}}.
\end{split}
\end{equation}
By the assumption on $p$, we have $0\leq \alpha\leq1$. By using the
fraction Liebniz rule for Sobolev spaces on the asymptotically conic
manifold \cite[Theorem 27]{CRT}, we have
\begin{equation}\label{5.5}
\begin{split}
\|F(u)\|_{L^{\tilde{q}'}_t\dot H^{\alpha}_{\tilde{r}'}}\leq
C\|u\|^{p-1}_{L^{q_0}_{t,z}}\|u\|_{L^{q_1}_t\dot
H^{\alpha}_{q_1}}\leq C^2(C\epsilon)^{p-1}\epsilon\leq
\frac{C\epsilon}2.
\end{split}
\end{equation}
A similar argument as above leads to
\begin{equation}\label{5.6}
\begin{split}
&\|\Phi(u_1)-\Phi(u_2)\|_{L^{q_1}([0,T];\dot
H^{\alpha}_{q_1}(M^\circ))\cap
L^{q_0}([0,T];L^{q_0}(M^\circ))}\\&\leq
C\|F(u_1)-F(u_2)\|_{L^{\tilde{q}'}_t\dot
H^{\alpha}_{\tilde{r}'}}\\&\leq
C^2(C\epsilon)^{p-1}\|u_1-u_2\|_{L^{q_1}([0,T];\dot
H^{\alpha}_{q_1}(M^\circ))\cap
L^{q_0}([0,T];L^{q_0}(M^\circ))}\\&\leq
\frac12\|u_1-u_2\|_{L^{q_1}([0,T];\dot
H^{\alpha}_{q_1}(M^\circ))\cap L^{q_0}([0,T];L^{q_0}(M^\circ))}.
\end{split}
\end{equation}
Therefore the solution map $\Phi$ is a contraction map on $X$ under
the metric $d(u_1,u_2)=\|u_1-u_2\|_{{L^{q_1}([0,T];\dot
H^{\alpha}_{q_1}(M^\circ))}\cap L^{q_0}([0,T];L^{q_0}(M^\circ))}$.
The standard contraction argument completes the proof of Theorem
\ref{thm1}.

\begin{center}

\end{center}
\end{document}